\documentclass[10pt,bezier]{article}
\usepackage{amsmath,amssymb,amsfonts,graphicx,caption,subcaption}
\usepackage[colorlinks,linkcolor=red,citecolor=blue]{hyperref}

\newtheorem{prethm}{{\bf Theorem}}[section]
\newenvironment{thm}{\begin{prethm}{\hspace{-0.5
em}{\bf.}}}{\end{prethm}}
\newtheorem{prepro}{{\bf Theorem}}

\newtheorem{precor}[prethm]{{\bf Corollary}}
\newenvironment{cor}{\begin{precor}{\hspace{-0.5
em}{\bf.}}}{\end{precor}}
\newtheorem{preconj}[prethm]{{\bf Conjecture}}

\newtheorem{preremark}[prethm]{{\bf Remark}}
\newenvironment{remark}{\begin{preremark}\em{\hspace{-0.5
em}{\bf.}}}{\end{preremark}}
\newtheorem{prelem}[prethm]{{\bf Lemma}}
\newenvironment{lem}{\begin{prelem}{\hspace{-0.5
em}{\bf.}}}{\end{prelem}}
\newtheorem{preque}[prethm]{{\bf Question}}

\newtheorem{preobserv}[prethm]{{\bf Observation}}

\newtheorem{preproposition}[prethm]{{\bf Proposition}}

\newtheorem{preproof}{{\bf Proof.}}
\newtheorem{preprooff}{{\bf Proof}}

\newenvironment{proof}[1]{\begin{preproof}{\rm
#1}\hfill{$\Box$}}{\end{preproof}}

\newtheorem{preproofs}{{\bf Second proof of }}

\newtheorem{preprooft}{{\bf Third proof of }}

\newtheorem{preproofF}{{\bf Proof of}}

\title{\bf\Large 
Toughness and the existence of tree-connected $\{f,f+k\}$-factors
}
\author{{\normalsize{\sc Morteza Hasanvand${}$} }\vspace{3mm}
\\{\footnotesize{${}$\it Department of Mathematical
 Sciences, Sharif
University of Technology, Tehran, Iran}}
{\footnotesize{}}\\{\footnotesize{  $\mathsf{morteza.hasanvand@alum.sharif.edu }$ }}}

\date{}
\begin{document}
\maketitle
\begin{abstract}{
Let $G$ be a graph and  let $f$ be a positive integer-valued function on $V(G)$ satisfying $2m\le f\le b$, where  $b$ and $m$ are two positive integers with $b\ge 4m^2$. In this paper, we show that  if $G$ is $b^2$-tough and $|V(G)|\ge b^2$, then it has an $m$-tree-connected  factor $H$ such that for each vertex $v$, $$d_H(v)\in \{f(v), f(v)+1\}.$$  Next, we generalize this result by  giving sufficient conditions for a tough graph to have a tree-connected factors $H$ such that for each vertex $v$, $d_H(v)\in \{f(v), f(v)+k\}$. As an application, we prove that every $64b(b-a)^2$-tough graph $G$ of order at least $b+1$ with  $ab|V(G)|$ even admits a  connected factor whose degrees lie in the set $\{a,b\}$, where  $a$ and $b$ are two  integers with $2\le a< b <  \frac{6}{5}a$. Moreover, we prove that every $16$-tough graph $G$ of order at least three admits a $2$-connected factor whose degrees lie in the set $\{2,3\}$, provided that $G$ has a $2$-factor with girth at least five. This result confirms a weaker version of a long-standing conjecture due to Chv\'atal (1973).
} {\small
}
\end{abstract}
%
%
%
%
%
%
%
%
%
%
%
%
%
%
\section{Introduction}
In this article, all graphs have no loop, but multiple edges are allowed and a simple graph is a graph without multiple edges.
 Let $G$ be a graph. 
The vertex set, the edge set, the minimum degree, the maximum degree, and the number of components of $G$ are denoted by $V(G)$, $E(G)$, $\delta(G)$, $\Delta(G)$, and $\omega(G)$, respectively. 
The degree $d_G(v)$ of a vertex $v$ is the number of edges of $G$ incident to $v$.
The set of edges of $G$ that are incident to $v$ is denoted by $E_G(v)$.
We denote by $d_G(C)$ the number of edges of $G$ with exactly one end in $V(C)$, where $C$ is a subgraph of $G$.
For a set $X\subseteq V(G)$, we denote by $G[X]$ the induced subgraph of $G$ with the vertex set $X$ containing
precisely those edges 
of $G$ whose ends lie in~$X$.
For a set $A$ of integers, an {\bf $A$-factor} is a spanning subgraph with vertex degrees in $A$. 
If for each vertex $v$, $A(v)$ is a set of  integers, an {\bf $A$-factor} is a spanning subgraph $F$ such that for each vertex $v$,
$d_F(v)\in A(v)$.
Let $F$ be a spanning subgraph of $G$.
For  an edge set  $E$, we denote by $F-E$ the graph obtained from $F$ by removing  the edges of $E$ from $F$.
Likewise, we denote by $F+E$ the graph obtained from $F$ by inserting  the edges of $E$ into $F$.
For convenience, we use $e$ instead of $E$ when $E=\{e\}$.
For two edge sets $E_1$ and $E_2$, we also use the notation $E_1+E_2$ for the union of them.
The graph obtained from $G$ by contracting any component of $F$ is denoted by $G/F$. 
The graph  $F$ is said to be {\bf trivial}, if it has no edge. 
Let $S\subseteq V(G)$. 
 The graph obtained from $G$ by removing all vertices of $S$ is denoted by $G\setminus S$.
Denote by $G\setminus [S,F]$ the graph obtained from $G$ by removing all edges incident to the vertices of $S$ except the edges of $F$. 
Note that while the vertices of $S$ are deleted in $G\setminus S$, no vertices are removed in $G\setminus [S, F]$.
Denote by $e_G(S)$ the number of edges of $G$ with both ends in $S$. 
Let $P$ be a partition of $V(G)$.
Denote by $e_G(P)$ the number of edges of $G$ whose ends lie in different parts of $P$.
A graph $G$ is called {\bf$m$-tree-connected}, if it has $m$ edge-disjoint spanning trees. 
In addition, an $m$-tree-connected graph $G$ is called {\bf minimally $m$-tree-connected},
 if $|E(G)|=m(|V(G)|-1)$.
In other words, for any edge $e$ of $G$, the graph $G\setminus e$ is not $m$-tree-connected.
The vertex set of any graph $G$ can be expressed uniquely as a disjoint union of vertex sets of some induced $m$-tree-connected subgraphs.
These subgraphs are called the
{\bf $m$-tree-connected components} of $G$.
For a graph $G$, we define the parameter $\Omega_m(G)=m|P|-e_G(P)$ to measure tree-connectivity, where $P$ is the unique partition of $V(G)$
obtained from the $m$-tree-connected components of $G$.
Note that $\Omega_1(G)$ is the same number of components of $G$, while $\Omega_m(G)$ is less or equal than the number of $m$-tree-connected components of $G$.
The definition implies that
the null 
graph $K_0$ with no vertices is not $m$-tree-connected and $\Omega_m(K_0)=0$.
In this paper, we assume that all graphs are nonnull, except for the graphs that obtained by removing vertices.
We say that a graph $F$ is {\bf $m$-sparse}, if $e_F(S)\le m|S|-m$, for all nonempty subsets $S\subseteq V(G)$.
It is not hard to check that $\Omega_m(F) < \Omega_m(F\setminus e)$, 
 for any edge $e$ of $F$.
Clearly, $1$-sparse graphs are forests.
%
A graph $G$ is $m$-sparse if and only if   its $m$-tree-connected components are minimally $m$-tree-connected.
Note that all maximal $m$-sparse factors of $G$ form the bases of a matroid, see~\cite{Edmonds-1970}.
Note also that several basic tools in this paper for working with sparse and tree-connected graphs can
be obtained using matroid theory. 
Let $t$ be a positive real number, 
a graph $G$ is said to be {\bf $t$-tough}, if $\omega(G\setminus S)\le \max\{1,\frac{1}{t}|S|\}$ for all $S\subseteq V(G)$. 
The {\bf bipartite  index $bi(G)$} of a graph $G$ is the smallest number of all $|E(G)\setminus E(H)|$ taken over all bipartite factors~$H$.
Throughout this article, all variables $k$ and $m$ are positive integers.
%
%
%
%
%
%
%

Recently, the present author~\cite{ ClosedWalks} investigated connected factors with small degrees and established
the following theorem. This result is an improvement several results due to Win (1989)~\cite{Win-1989}, Ellingham and Zha~(2000)~\cite{Ellingham-Zha-2000}, and Ellingham, Nam, and Voss~(2002)~\cite{Ellingham-Nam-Voss-2002}.
\begin{thm}{\rm(\cite{ClosedWalks})}\label{intro:thm:c=3:G[S]}
{Let $G$ be a graph with a factor $F$ of which every component contains at least $c$ vertices with $c \ge 2$.
Let $h$ be a nonnegative integer-valued function on $V(G)$.
If for all $S\subseteq V(G)$,
$$ \omega (G\setminus S) < \sum_{v\in S}\big(\frac{c}{2c-2}h(v)-\frac{1}{c-1}\big)+2+\frac{1}{c-1}\omega(G[S]),$$
then $G$ has a connected factor $H$ containing $F$  such that for each vertex $v$, $d_H(v) \le h(v)+d_F(v)$.
}\end{thm}

In this paper, we generalize Theorem~\ref{intro:thm:c=3:G[S]} to the following tree-connected version with more complicated arguments. 
The special case $c=2$ of this theorem was former studied in~\cite{ClosedTrails}.
\begin{thm}\label{intro:thm:gen:tough-enough}
{Let $G$ be a graph with a factor $F$. Assume  that
for every $m$-tree-connected component $C$ of $F$, $|V(C)| +\frac{c-1}{2m}d_F(C) \ge  c$ in which $c\ge 2$.
Let $h$ be an integer-valued function on $V(G)$.
If for all $S\subseteq V(G)$,
$$ \Omega_m(G\setminus S) < \sum_{v\in S}\big(\frac{c}{2c-2}h(v)-\frac{m}{c-1}\big)+m+1+\frac{1}{c-1}\Omega_m(G[S]),$$
then $G$ has an $m$-tree-connected factor $H$ containing $F$ 
such that for each vertex $v$, 
$d_H(v)\le h(v)+d_F(v) .$
}\end{thm}

In 1973 Chv\'atal~\cite{Chvatal-1973} conjectured that there exists a positive real number $t_0$ such that every $t_0$-tough graph of order at least three admits a Hamiltonian cycle (connected $2$-factor).
In 1989 Win~\cite{Win-1989} gave the first step forward to confirm this conjecture by proving that every $1$-tough graph of order at least three admits a connected $\{1,2,3\}$-factor.
In 2000 Ellingham and Zha~\cite{Ellingham-Zha-2000} improved Win's result
by proving that every $4$-tough graph of order at least three admits a connected $\{2,3\}$-factor.
In Subsection~\ref{subsec:2-connected}, we prove the following stronger assertion for graphs with higher toughness.
\begin{thm}
{Let $G$ be a graph of order at least three. If $G$ is $16$-tough, then it admits a $2$-connected $\{2,3\}$-factor, provided that $G$ admits a $2$-factor with girth at least five.
}\end{thm}

In 1990  Katerinis~\cite{Katerinis-1990} formulated the following sufficient toughness condition for the existence of $f$-factors.
Next, some sufficient toughness-type conditions for the existence of connected $\{f,f+1\}$-factors were investigated in~\cite{Ellingham-Nam-Voss-2002,Ellingham-Zha-2000,iso-tough}.
\begin{thm}{\rm (\cite{Katerinis-1990}, see~\cite{iso-tough})}\label{intro:thm:f} 
{Let $G$ be a  graph,  let $b$ be a positive integer with $b\ge 2$, and  let $f$ be a positive integer-valued function on $V(G)$ with  $f\le b$. If $G$ is $b^2$-tough and $|V(G)|\ge b^2$, then it has a factor $F$ such that for all vertices $v$, $d_H(v)=f(v)$, except possibly for a vertex $u$ satisfying $d_H(u)=f(u)+1$.
}\end{thm}

 In Section~\ref{sec:f,f+1}, we apply Katerinis's result to give  a sufficient toughness condition for the existence of $m$-tree-connected$\{f,f+1\}$-factors as mentioned in the abstract. Moreover, we apply it together with the following recent result to give a sufficient toughness condition for a graph to have a tree-connected $\{f,f+k\}$-factor.

\begin{thm}{\rm (\cite{complementary})}\label{intro:thm:complementary}
{Let $G$ be a $(2m+2m_0+6k)$-tree-connected graph satisfying $bi(G)\ge k-1$ and $k>  m+m_0\ge 0$, and let $f$ be a positive integer-valued function on $V(G)$. 
If for each vertex $v$, $f(v)+m_0\le \frac{1}{2}d_G(v)\le f(v)+k-m$,
then $G$ has an $m$-tree-connected factor $H$ such that its complement is $m_0$-tree-connected and for each vertex $v$,
$$d_H(v)\in \{f(v),f(v)+k\},$$ provided that $(k-1)\sum_{v\in V(G)}f(v)$ is even.
}\end{thm}
%
%
%
%
%
\section{Preliminary results}
\label{sec:Generalizations}
Here, we state a fundamental theorem for finding tree-connected factors with small degrees which is similar to Theorem 3.1 in~\cite{ClosedTrails} which  provides a tree-connected version for Theorem 1 in~\cite{Ellingham-Zha-2000}.
Before doing so, let us state some basic tools for working with sparse graphs and tree-connected graphs which are as well-known in terms of matroid theory, see \cite{ClosedTrails}.
\begin{lem}\label{lem:add-e}
{Let $G$ be a graph with an $m$-sparse factor $F$.
If $e\in E(G)\setminus E(F)$ joins different tree-connected components of $F$, then $F+e$ is still $m$-sparse.
}\end{lem}
\begin{lem}\label{lem:xGy-exchange}
{Let $G$ be a graph with an $m$-sparse factor $F$.
If $xy\in E(G)\setminus E(F)$ and $Q$ is a minimal $m$-tree-connected subgraph of $F$ including $x$ and $y$, then for every $e \in E(Q)$, the graph $F-e+xy$ remains $m$-sparse.
}\end{lem}
\begin{lem}\label{lem:X-Y}
{Let $G$ be a graph with two $m$-sparse factors $F$ and $F_0$. If for a vertex set $X$,  
$F[X]$ is  $m$-tree-connected, then $F-E(F[X])+E(F_0[X])$ remains $m$-sparse.
}\end{lem}
\begin{proof}
{Set $H=F-E(F[X])+E(F_0[X])$ and let $A$ be a nonempty subset of $V(H)$.
If $X\cap A$ is empty, then $e_{H}(A)= e_{F}(A)\le m(|A|-1)$.
So, suppose that $X\cap A$ is not empty.
Obviously, $e_{F}(A)-e_{F}(X\cap A) \le e_{F}(X\cup A)-e_F(X)$.
Since $e_{F}(X)=m(|X|-1)$ and $e_{F}(X\cup  A)\le m(|X| +|A|-|X\cap A|-1)$, we must have 
$e_{F}(A)-e_{F}(X\cap A)\le m(|A|-|X\cap A|)$.
Since $e_{F_0}(X\cap A )\le m(|A\cap X|-1)$, we therefore have $e_{H}(A)= e_{F}(A)-e_{F}(X\cap A)+e_{F_0}(X\cap A )\le m(|A|-1)$. Hence $H$ is $m$-sparse and the proof is completed.
}\end{proof}
Now, we are ready to prove the main result of this section.
\begin{thm}\label{thm:final:minor-improvement}
{Let $G$ be a graph with an $m$-sparse factor $F$ 
and let $h$ be a nonnegative integer-valued function on $V(G)$.
 If $H$ is an $m$-sparse factor of $G$ containing $F$ with $te(H,h+d_F)=0$ and with the minimum $\Omega_m(H)$, then there exists a subset $S$ of $V(G)$ with the following properties:
\begin{enumerate}{
\item $\Omega_m(G\setminus [S,F])=\Omega_m(H\setminus [S,F])$.\label{condition:final:1}
\item For each vertex $v$ of $S$, $d_H(v)= h(v)+d_F(v)$.\label{condition:final:2}
}\end{enumerate}
}\end{thm}
\begin{proof}
{Define $V_0=\emptyset $. 
For any $S\subseteq V(G)$ and  $u\in V(G)\setminus S$, 
let $\mathcal{A}(S, u)$ be the set of all $m$-sparse factors 
$H^\prime $ of $G$ containing $F$
with $te(H',h+d_F)=0$ such that 
 $\Omega_m(H')=\Omega_m(H)$, $H'[X]$ is $m$-tree-connected, and
$H^\prime$ and $H$ have the same edges, except for some of the edges of $G$ whose ends are in $X$, 
where $H[X]$ is the $m$-tree-connected component of $H\setminus [S,F]$ containing $u$.
Now, for each integer $n$ with $n\ge 2$,  recursively define $V_n$ as follows:
$$V_n=V_{n-1} \cup \{\, v\in V(G)\setminus V_{n-1} \colon \, d_{H^\prime }(v)= h(v)+d_F(v) \text{\, for all\, }H^\prime \in \mathcal{A}(V_{n-1},v)\,\}.$$
 Now, we prove the following claim.
%
\vspace{2mm}\\
{\bf Claim.} 
Let $x$ and $y$ be two vertices in different m-tree-connected components of $H\setminus [V_{n-1},F]$.
If $xy\in E(G)\setminus E(H)$, then $x\in V_{n}$ or $y\in V_{n}$.
\vspace{2mm}\\
{\bf Proof of Claim.} 
By induction on $n$.
Suppose, to the contrary,  that vertices $x$ and $y$ are in different $m$-tree-connected components of $H\setminus [V_{n-1},F]$,
$xy\in E(G)\setminus E(H)$, 
and $x,y\not \in V_{n}$. 
Let $X$ and $Y$ be the vertex sets  of the $m$-tree-connected components of $H\setminus [V_{n-1},F]$
containing $x$ and $y$, respectively. 
Since $x,y\not\in V_{n}$,
 there exist
 $H_x\in \mathcal{A}(V_{n-1},x)$ and
 $H_y\in \mathcal{A}(V_{n-1},y)$ with $d_{H_x}(x)< h(x)+d_F(x)$ and $d_{H_y}(y)< h(y)+d_F(y)$. 
 For $n=1$, define $H'$ to be the factor of $G$ containing $F$ with
 $$E(H')=E(H)+xy
-E(H[X])+E(H_x[X])
-E(H[Y])+E(H_y[Y]).$$
Since  $xy$ joins different $m$-tree-connected components of $H$, by Lemma~\ref{lem:add-e}, $H'$ is still $m$-sparse and so
$\Omega_m(H') < \Omega_m(H)$.
Since  $te (H',h+d_F)=0$, we arrive at a contradiction.
Now, suppose $n\ge 2$.
By the induction hypothesis,
 $x$ and $y$ are in the same $m$-tree-connected component 
of $H\setminus [V_{n-2},F]$ with the vertex set $Z$ so that $X\cup Y \subseteq Z$.
Let $Q$ be a minimal $m$-tree-connected subgraph
of $H$ including $x$ and $y$.
Notice that the vertices of $Q$ lie in $Z$ and also $Q$
includes at least a vertex $z$ of $Z\cap  V_{n-1}$ with $zz' \in E(Q)\setminus E(F)$.
By Lemma~\ref{lem:xGy-exchange},
the graph $H-zz'+xy$ is still $m$-sparse.
Now, let $H'$ be the factor of $G$  containing  $F$ with
 $$E(H')=E(H)-zz'+xy
-E(H[X])+E(H_x[X])
-E(H[Y])+E(H_y[Y]).$$
According to Lemma~\ref{lem:X-Y}, one can easily conclude that $H'$ is also $m$-sparse. 
Since $|E(H')|=|E(H)|$, we must have  $\Omega_m(H') = \Omega_m(H)$.
For each $v\in V(H')$, we have
$$d_{H'}(v)\le  
 \begin{cases}
d_{H_v}(v)+1,	&\text{if  $v\in \{x,y\}$};\\
d_{H}(v),	&\text{if  $v=z'$},
\end {cases}
\quad \text{ and }\quad
 d_{H'}(v)= 
 \begin{cases}
d_{H_x}(v),	&\text{if  $v\in X\setminus \{x,z'\}$};\\
d_{H_y}(v),	&\text{if  $v\in Y\setminus \{y,z'\}$};\\
d_{H}(v),	&\text{if  $v\notin X\cup Y\cup \{z,z'\}$}.
\end {cases}$$
It is not hard to check that $d_{H'}(z)<d_{H}(z)\le h(z)+d_F(z)$ and $H'$ lies in $ \mathcal{A}(V_{n-2}, z)$.
Since $z\in V_{n-1}\setminus V_{n-2}$, we arrive at a contradiction.
Hence the claim holds.
%
%
%

Obviously, there exists a positive integer $n$ such that and $V_1\subseteq \cdots\subseteq V_{n-1}=V_{n}$.
 Put $S=V_{n}$. 
For each $v\in V_i\setminus V_{i-1}$, 
we have $H\in \mathcal{A}(V_{i-1},v)$ and so $d_H(v)= h(v)+d_F(v)$. 
This establishes  Condition~\ref{condition:final:2}.
 Because $S=V_{n}$, 
the previous claim implies Condition~\ref{condition:final:1} and completes the proof.
}\end{proof}
\section{Highly tree-connected factors with bounded degrees}
Our aim in this section is to prove Theorem~\ref{intro:thm:gen:tough-enough} with a more stronger version.
Before formulating the main result, we shall begin with some lemmas that allows us to make the proof shorter.
\subsection{Comparing tree-connectivity of $G \setminus [S,F]$ and $G \setminus S$}

In this subsection, we shall compare two parameters
$\Omega_m(G \setminus S)$
and $\Omega_m(G \setminus [S,F])$.
Before stating the main comparison, let us make the following simpler version.
Note that $\Omega_m(H\setminus [S,F])=\Omega_m(H\setminus S)+m|S|$ when $F$ is the trivial factor.
\begin{lem}\label{lem:GS,G[S,F]}
{Let $G$ be a graph with a factor $F$.
If $S\subseteq V(G)$, then
$$\Omega_m(G\setminus [S, F]) 
\le \Omega_m(G\setminus S)+mt_s-i_{s}.$$
where $t_s$ is the number  of $m$-tree-connected components of $G\setminus [S,F]$ whose vertices entirely lie in $S$, and
 $i_{s}$ is the number of edges of $F$ incident to vertices in $S$ joining different $m$-tree-connected components of $G\setminus [S,F]$.
}\end{lem}
\begin{proof}
{Define $P$ to be the partition of $V(G)$ obtained from the $m$-tree-connected components of $G\setminus [S,F]$.
Let $P_s=\{A\in P: A\subseteq S\}$ and $P'=\{A\in P: A\not \in P_s\}$, and also $P_0=\{A\setminus S: A\in P'\}$.
For every $A\in P$, define $P_A$  to be the partition of $A\setminus S$ obtained from the $m$-tree-connected components of $G[A]\setminus S$. Also, define $e_A$ to be the number of edges with both ends in $A\setminus S$ joining different $m$-tree-connected components of $G[A]\setminus S$.
Since for every $A\in P'$, $m(|P_A|-1)\ge e_A$, we must have 
$$\Omega_m(G\setminus [S, F])=
m|P_s|+m|P'|- e_{G\setminus [S, F]}(P)\le
m|P_s|+ \sum_{A\in P'}(m|P_A|-e_A)- e_{G\setminus [S, F]}(P),$$
which implies that  
$$\Omega_m(G\setminus [S, F]) \le  
m|P_s|+m\sum_{A\in P'}|P_A|-e_{G\setminus S}(P_0)-i_s\le 
mt_s+\Omega_m(G\setminus S)-i_s.$$
Hence the assertion holds.
}\end{proof}
The following theorem  is a generalization of 
Lemma 5.1 in \cite{ClosedWalks} and provides a useful relationship between two parameters
$\Omega_m(G \setminus S)$
and $\Omega_m(G \setminus [S,F])$.
\begin{lem}\label{lem:gen}
{Let $G$ be a graph with a factor $F$.
Let $c\in (1,\infty)$ be a real number and let $\xi:V(G)\rightarrow [0,m]$ be a real function 
in which for every $m$-tree-connected component $C$ of $F$, 
$\sum_{v\in V(C)}\xi(v)\ge m-\frac{m}{c-1}(|V(C)-1|)-\frac{1}{2}d_F(C)$.
If $S\subseteq V(G)$, then
$$ \Omega_m(G\setminus [S,F])\le \Omega_m(G\setminus S)+
\frac{1}{(c-1)}e_F(S)+\sum_{v\in S}\xi(v).$$
In addition, $\Omega_m(G\setminus [S,F])\le\Omega_m(G\setminus S)+ e_F(S)+\sum_{v\in S}\max\{0,m- d_F(v)\}$.
}\end{lem}
\begin{proof}
{We denote by  $P_s$ be the set of vertex sets of all $m$-tree-connected components of $G\setminus [S,F]$ whose vertices entirely lie in $S$. In the first statement, since $e_F(A)\ge m(|A|-1)$ for every $A\in P_s$, we must have
$$m-\frac{1}{2}d_F(A)\le \frac{m(|A|-1)}{(c-1)}+ \sum_{v\in A}\xi(v)\le
 \frac{1}{(c-1)}e_F(A)+\sum_{v\in A}\xi(v),
$$
which implies that
$$m|P_s|-i_s\le \sum_{A\in P_s}(m-\frac{1}{2}d_F(A))\le
 \frac{1}{(c-1)}\sum_{A\in P_s}e_F(A)+\sum_{A\in P_s}\sum_{v\in A}\xi(v)\le \frac{1}{(c-1)}e_F(S)+\sum_{v\in S}\xi(v). 
$$
Now, let us define $P'_s$ to be the set of all $A\in P_s$ with $|A|\ge 2$ and define  $S'$ to be the set of all $v\in S$ with $\{v\}\in P_s$.
In the second statement, since $e_F(A)\ge m(|A|-1)\ge m$ for every $A\in P'_s$, we must have
$$m|P'_s|+m|S'|-i_s\le\sum_{A\in P'_s}e_F(A)+ 
m|S'|-\sum_{v\in S'} d_F(v)+e_F(S')
\le \sum_{v\in S'} (m-d_F(v))+e_F(S),$$
which implies that
$$m|P_s|-i_s\le 
\sum_{v\in S} \max\{0,m- d_F(v)\}+e_F(S).$$
Hence the two assertions follow from Lemma~\ref{lem:GS,G[S,F]}.
}\end{proof}
%
%
%
%
%
%
%
%
%
\subsection{Extending  $m$-sparse factors}
\label{subsec:Toughness:Factors}
The following lemma is a common generalization of Lemma 3.1 in \cite{ClosedWalks} and
Lemma 4.2~in~\cite{ClosedTrails}.
\begin{lem}\label{lem:spanningforest:gen}
{Let $H$ be an $m$-sparse graph with a factor $F$. If $S\subseteq V(H)$ and $\mathcal{F}=H\setminus E(F)$, then 
$$\sum_{v\in S}d_{\mathcal{F}}(v)=\Omega_m(H\setminus [S,F])-\Omega_m(H)+e_\mathcal{F}(S).$$
}\end{lem}
\begin{proof}
{By induction on the number of edges of $\mathcal{F}$ which are incident to the vertices in $S$. If there is no edge of  $\mathcal{F}$ incident to a vertex in $S$, then the proof is clear. Now, suppose that there exists an edge $e=uu'\in E(\mathcal{F})$ with $|S\cap \{u,u'\}|\ge 1$. Hence
\begin{enumerate}{
\item $\Omega_m(H)=\Omega_m (H\setminus e)-1,$
\item $\Omega_m(H\setminus [S,F])= \Omega_m((H\setminus e)\setminus [S,F]),$

\item $e_\mathcal{F}(S)=e_{\mathcal{F}\setminus e}(S)+|S\cap \{u,u'\}|-1,$
\item $\sum_{v\in S}d_{\mathcal{F}}(v)=\sum_{v\in S}d_{\mathcal{F}\setminus e}(v)\,+|S\cap \{u,u'\}|.$
}\end{enumerate}
Therefore, by the induction hypothesis on $H\setminus e$ with the factor $F$ the lemma holds.
}\end{proof}
The following theorem is essential in this section.
\begin{thm}\label{thm:eF(S):base}
{Let $G$ be a graph with a factor $F$ and let $h$ be a nonnegative integer-valued function on $V(G)$.
Define $\mathcal{A}$ to be the set of all $m$-sparse factors $\mathcal{F}$ of $G$ whose edges join different $m$-tree-connected components of $F$ and for each vertex $v$, $d_{\mathcal{F}}(v)\le h(v)$. 
If for all $S\subseteq V(G)$, 
$$ \Omega_m(G\setminus [S,F]) \le \sum_{v\in S}h(v)+m-\max_{\mathcal{F} \in \mathcal{A}}e_\mathcal{F}(S),$$
then $G$ has an $m$-tree-connected factor $H$ containing $F$ 
such that for each vertex $v$, 
$d_H(v)\le h(v)+d_F(v)$
}\end{thm}
\begin{proof}
{First, suppose that $F$ is $m$-sparse.
Let $H$ be an $m$-sparse factor of $G$ containing $F$ with $te(H,h+d_F)=0$ and with the minimum $\Omega_m(H)$.
Define $S$ to be a subset of $V (G)$ with the properties described in Theorem~\ref{thm:final:minor-improvement}.
Put $\mathcal{F}=H\setminus E(F)$ so that $\mathcal{F}\in \mathcal{A}$.
By Lemma~\ref{lem:spanningforest:gen} and Theorem~\ref{thm:final:minor-improvement},
$$\sum_{v\in S} h(v) = \sum_{v\in S} d_\mathcal{F}(v) = 
 \Omega_m(H\setminus [S, F]) -\Omega_m(H)+ e_{\mathcal{F}}(S),$$
and so
$$\Omega_m(H) = \Omega_m(G\setminus [S, F]) + e_{\mathcal{F}}(S)- \sum_{v\in S} h(v)\le m.$$
Hence $\Omega_m(H) =m$ and the theorem holds.
Now, suppose that $F$ is not $m$-sparse.
Remove some of the edges of the $m$-tree-connected components of $F$ until the resulting graph $F'$
becomes $m$-sparse such that their $m$-tree-connected components have the same vertices. 
It is enough, now, to apply the theorem on $F'$ and finally add the edges of $E(F)\setminus E(F')$ 
to that explored $m$-tree-connected factor. 
}\end{proof}
The following corollary shows an interesting  application of Theorem~\ref{thm:eF(S):base}.
The special case $m=1$ of this result says that every linear forest of every $1$-tough graph can be extend to a spanning tree with maximum degree at most $3$.
\begin{cor}\label{cor:sparse:linear-forest}
{Let $G$ be a graph with an $m$-sparse factor $F$ satisfying $\Delta(F)\le 2m$. Then $F$ can be extended to an $m$-tree-connected factor $H$  satisfying $\Delta(H)\le 2m+1$, if for all $S\subseteq V(G)$, $$ \Omega_m(G\setminus S) \le \frac{1}{2}|S|+m.$$ 
}\end{cor}
\begin{proof}
{For each vertex $v$, define $f(v)=2m+1-d_F(v)$ so that $f(v)\ge 1$.
Let $S$ be a subset of $V(G)$ and 
let $P$ be the set of vertex sets of $m$-tree-connected components of $G\setminus [S,F]$ 
whose vertices entirely lie in $S$. 
By the assumption and Lemma~\ref{lem:GS,G[S,F]},
$$\Omega_m(G\setminus [S, F]) 
\le \Omega_m(G\setminus S)+m|P|-\frac{1}{2}\sum_{A\in P}d_F(A)\le
 \frac{1}{2}|S|+m-\sum_{A\in P}(\frac{1}{2}d_F(A)-m).$$
Since $F$ is $m$-sparse, for every $A\in P$, we must have
$$0\le m(|A|-1)-e_{F}(A)=
\sum_{v\in A}(m-\frac{1}{2}d_F(v))+\frac{1}{2}d_F(A)-m
=\sum_{v\in A} (\frac{1}{2} f(v)-\frac{1}{2})+\frac{1}{2}d_F(A)-m,$$
which implies that
$$\frac{1}{2}|S|-\sum_{A\in P}(\frac{1}{2}d_F(A)-m) \le 
\frac{1}{2}|S|+
\sum_{v\in A\in P}(\frac{1}{2}f(v)-\frac{1}{2}) \le 
\sum_{v\in S}\frac{1}{2}f(v).$$
Let $\mathcal{F}$ be an arbitrary $m$-sparse factor  of $G$ whose edges join different $m$-tree-connected components of $F$ and for each vertex $v$, $d_{\mathcal{F}}(v)\le f(v)$. 
Obviously, $e_\mathcal{F}(S)\le \sum_{v\in S}\frac{1}{2}d_{\mathcal{F}}(v)\le  \sum_{v\in S}\frac{1}{2}f(v)$. Therefore,
$$ \Omega_m(G\setminus [S,F]) \le \sum_{v\in S} \frac{1}{2}f(v)+m \le 
\sum_{v\in S} f(v)+m-e_{\mathcal{F}(S)}.$$
Hence the assertion follows from Theorem~\ref{thm:eF(S):base}.
}\end{proof}
\subsection{Extending  factors with large tree-connected components}
\label{subsec:-strongly tough enough graphs}
The following theorem provides a stronger version for Theorem~\ref{intro:thm:gen:tough-enough} 
which is motivated by the main result in \cite{ClosedWalks}.
\begin{thm}\label{thm:gen:tough-enough}
{Let $G$ be a graph with a factor $F$ and let $h$ be a nonnegative integer-valued function on $V(G)$.
Let $c\in [2,\infty)$ be a real number and let $\xi:V(G)\rightarrow [0,m]$ be a real function in which for every $m$-tree-connected subgraph $C$ of $F$, 
 $\sum_{v\in V(C)}\xi(v)\ge m-\frac{m}{c-1}(|V(C)|-1)-\frac{1}{2}d_F(C)$.
If for all $S\subseteq V(G)$, 
$$ \Omega_m (G\setminus S) <
\sum_{v\in S}\big(h(v)-\frac{m}{c-1}-\xi(v)\big)+m+1+\frac{1}{c-1}\Omega_m(G[S])-
\frac{c-2}{c-1}\min\{\lfloor\sum_{v\in S}\frac{h(v)}{2}\rfloor,m|S|-\Omega_m(G[S])\},$$
then $G$ has an $m$-tree-connected factor $H$ containing $F$ 
such that for each vertex $v$, 
$d_H(v)\le h(v)+d_F(v)$
}\end{thm}
\begin{proof}
{First, suppose that $F$ is $m$-sparse.
Let $H$ be an $m$-sparse factor of $G$ containing $F$ with $te(H,h+d_F)=0$ and with the minimum $\Omega_m(H)$.
Define $S$ to be a subset of $V (G)$ with the properties described in Theorem~\ref{thm:final:minor-improvement}.
Put $\mathcal{F}=H\setminus E(F)$.
By Lemma~\ref{lem:spanningforest:gen} and Theorem~\ref{thm:final:minor-improvement},
$$\sum_{v\in S} h(v) = \sum_{v\in S} d_\mathcal{F}(v) = 
 \Omega_m(H\setminus [S, F]) -\Omega_m(H)+ e_{\mathcal{F}}(S),$$
and so
\begin{equation}\label{eq:thm:xi:1}
 \Omega_m(H) = \Omega_m(G\setminus [S, F]) - \sum_{v\in S} h(v) + e_{\mathcal{F}}(S).
\end{equation}
Also, by Lemma~\ref{lem:gen},
\begin{equation}\label{eq:thm:xi:2}
\Omega_m(G\setminus [S,F]) \le 
\Omega_m(G\setminus S)+\frac{1}{c-1}e_F(S)+\sum_{v\in S}\xi(v).
\end{equation}
Since 
$e_\mathcal{F}(S)+e_F(S)=e_{H}(S) \le m|S|-\Omega_m(G[ S])$,
$$e_\mathcal{F}(S)+\frac{1}{c-1}e_F(S)\le \frac{c-2}{c-1}e_\mathcal{F}(S)+\frac{1}{c-1}(m|S|-\Omega_m(G[S])).$$
 In addition, since $e_\mathcal{F}(S) \le\frac{1}{2}\sum_{v\in S}d_{\mathcal{F}}(v) =\lfloor\frac{1}{2}\sum_{v\in S}h(v)\rfloor$, we must have
\begin{equation}\label{eq:thm:xi:3}
e_\mathcal{F}(S)+\frac{1}{c-1}e_F(S)
\le \frac{c-2}{c-1}\min\{\lfloor\frac{1}{2}\sum_{v\in S}h(v)\rfloor, m|S|-\Omega_m(G[S])\}+\frac{1}{c-1}(m|S|-\Omega_m(G[S])).
\end{equation}
Therefore, Relations~(\ref{eq:thm:xi:1}),~(\ref{eq:thm:xi:2}), and~(\ref{eq:thm:xi:3}) can conclude that
\begin{equation*}
\Omega_m(H) \le
\Omega_m(G\setminus S)-
\sum_{v\in S}\big(h(v)-\frac{m}{c-1}-\xi(v)\big)-\frac{1}{c-1}\Omega_m(G[S])+
\frac{c-2}{c-1}\min\{\lfloor\sum_{v\in S}\frac{h(v)}{2}\rfloor,|S|-\Omega_m(G[S])\}
< 2.
\end{equation*}
Hence $\Omega_m(H) =m$ and the theorem holds.
Now, suppose that $F$ is not $m$-sparse.
Remove some of the edges of the $m$-tree-connected components of $F$ until the resulting graph $F'$
becomes $m$-sparse such that their $m$-tree-connected components have the same vertices. 
For every $m$-tree-connected component $F[A]$ of $F'$, we still have $d_{F'}(A)=d_F(A)$.
It is enough, now, to apply the theorem on $F'$ and finally add the edges of $E(F)\setminus E(F')$ 
to that explored $m$-tree-connected factor. 
}\end{proof}
\begin{remark}
{In  Theorem~\ref{thm:gen:tough-enough} we could select $c$ and $\xi$ depending  on $S$.
In addition, we could replace  the estimation $\frac{c-2}{c-1}\min\{\lfloor\sum_{v\in S}\frac{h(v)}{2}\rfloor,m|S|-\Omega_m(G[S])-e_F(S)\}$ when $F$ is $m$-sparse.
}\end{remark}
When $\xi=0$, the theorem becomes much simpler as the next corollary.
\begin{cor}\label{cor:gen:tough-enough}
{Let $G$ be a graph with a factor $F$ of which every $m$-tree-connected component $C$ satisfies
$|V(C)|+\frac{c-1}{2m}d_F(C) \ge c$  
 and $c \ge 2m+1$.
Let $h$ be a nonnegative integer-valued function on $V(G)$.
If for all $S\subseteq V(G)$,
$$ \Omega_m(G\setminus S) < \sum_{v\in S}\big(\frac{c}{2c-2}h(v)-\frac{m}{c-1}\big)+m+1+\frac{1}{c-1}\Omega_m(G[S]),$$
then $G$ has an $m$-tree-connected factor $H$ containing $F$ 
such that for each vertex $v$, 
$d_H(v)\le h(v)+d_F(v) .$
}\end{cor}
\begin{proof}
{By the assumption, for every $m$-tree-connected component $C$ of $F$,  we must  have 
$0\ge m-\frac{m}{c-1}(|V(C)|-1)-\frac{1}{2}d_F(C)$. Thus it is enough to apply Theorem~\ref{thm:gen:tough-enough} with $\xi(v)=0$.
}\end{proof}
The following corollary is an improvement of Corollary~\ref{cor:sparse:linear-forest} for graphs  with higher toughness.
\begin{cor}\label{cor:path-cycle}
{Let $G$ be a simple graph with a factor $F$ satisfying $\Delta(F)\le 2m$. 
Then $F$ can be extended to an $m$-tree-connected factor $H$ satisfying $\Delta(H)\le 2m+1$,
if for all $S\subseteq V(G)$, $$\Omega_m(G\setminus S)\le \frac{1}{4m}|S|+m.$$
}\end{cor}
\begin{proof}
{For each vertex $v$, define $h(v)=2m+1-d_F(v)$
and $\xi(v)=\frac{1}{|V(C_v)|}\max\{0,m-\frac{1}{2}(|V(C_v)|-1)-\frac{1}{2}d_F(v,C_v)\}$, where $C_v$ is the
 $m$-tree-connected component of $F$ including $v$, and also $d_F(v,C_v)$ denotes the number of edge of $F$ incident to $v$ having exactly one vertex in $C$. According to this definition, it is easy to see that $\sum_{v\in V(C)}\xi(v)\ge m-\frac{1}{2}(|V(C)|-1)-\frac{1}{2}d_F(C)$, where  $C$ is an  $m$-tree-connected component of $F$.
Since $F$ has no multiple edges, we must have $d_F(v, C)\ge d_F(v)-(|V(C_c)|-1)$.
Thus $$\xi(v)\le  \frac{1}{|V(C_v)|}\max\{0, m-\frac{1}{2}d_F(v)\}=\frac{1}{|V(C_v)|}(m-\frac{1}{2}d_F(v)).$$
This implies that $$\frac{2m+1}{4m}h(v)-\frac{1}{2}-\xi(v)\ge 
 \frac{2m+1}{4m}(2m+1-d_F(v))-\frac{1}{|V(C_v)|}(m-\frac{1}{2}d_F(v))\ge  \frac{1}{4m}.$$
 Therefore, by applying Theorem~\ref{thm:gen:tough-enough} with $c=2m+1$, the factor $F$ can be extended to an $m$-tree-connected factor $H$ such that for each vertex $v$,  $d_H(v)\le h(v)+d_F(v)= 2m+1$. Hence the assertion holds.
}\end{proof}
\section{Highly tree-connected $\{f,f+1\}$-factors including given $f$-factors}
\subsection{Tree-connected $\{f,f+1\}$-factors}
\label{sec:f,f+1}
When we consider the special cases $h(v)\le 1$, Corollary~\ref{cor:gen:tough-enough} becomes simpler as the following result.
This theorem plays an essential role of the results of  this section.
\begin{thm}\label{thm:tough:1-Extend}
{Let $G$ be a graph with a factor $F$ of which every $m$-tree-connected component $C$  satisfies
$|V(C)|+\frac{c-1}{2m}d_F(C) \ge c$  
 and $c \ge 2m+1$.
If for all $S\subseteq V(G)$,
$$\omega(G\setminus S) \le 
\frac{c-2m}{2m(c-1)}|S|+1$$
then $G$ has an $m$-tree-connected factor $H$ containing $F$ 
such that for each vertex $v$, 
$d_H(v) \in \{d_F(v), d_F(v)+1\}$, and $d_H(u)=d_F(u)$ for an arbitrary given vertex $u$.
}\end{thm}
\begin{proof}
{Let  $G'$ be the union of $m$ copies of $G$ with the same vertex set.
It is easy to check that $\frac{1}{m}\Omega_m(G'\setminus S) =\omega(G\setminus S)$ for every $S\subseteq V(G)$. 
Define $h(u)=0$ and $h(v)=1$ for each vertex $v$ with $v\neq u$.
By Corollary~\ref{cor:gen:tough-enough}, the graph 
$G'$ has an $m$-tree-connected factor $H$ containing $F$ such that for each vertex $v$, 
$d_{H}(v)\le h(v)+d_F(v)$.
According to the construction,  the graph $H$ must have no multiple edges of $E(G')\setminus E(F)$.
Hence $H$ itself is a factor of $G$ and the proof is completed.
}\end{proof}
The following corollary shows an application of Theorem~\ref{thm:tough:1-Extend}.
\begin{cor}\label{cor:f,f+1:lower-toughness}
{Let $G$ be a simple graph and let $f$ be a positive integer-valued function on $V(G)$ with $f\ge a \ge 2m$, where $a$ is a positive integer. Assume that $G$ contains a factor $F$ such that for all vertices $v$, $d_H(v)=f(v)$, except possibly for a vertex $u$ with $d_H(u)=f(u)+1$. If for all $S\subseteq V(G)$,
$$\omega(G\setminus S) \le 
\frac{a +1-2m}{2ma}|S|+1,$$ then $G$ has an $m$-tree-connected  factor $H$ containing $F$ 
such that for each vertex $v$,  $d_H(v)\in \{f(v),f(v)+1\}$.
}\end{cor}
\begin{proof}
{Since $F$ is simple, it is easy to check that $d_F(C)\ge \delta (F)-(|V(C)|-1)$, where $C$ is an $m$-tree-connected subgraph.
Since $a \ge 2m$, we must have 
$|V(C)|+\frac{a}{2m}d_F(C)\ge |V(C)| + (\delta(F)+1-|V(C)|)\ge a+1$.
Thus by applying Theorem~\ref{thm:tough:1-Extend} with $c=a+1$, 
the graph $G$ has an $m$-tree-connected factor $H$ containing $F$ 
such that  for each vertex $v$, $d_H(v) \le d_F(v)+1$, and also $d_H(u)=d_F(u)$.
This implies that $H$ is an $m$-tree-connected $\{f,f+1\}$-factor. 
}\end{proof}
\begin{cor}\label{cor:f,f+1}
{Let $G$ be a graph, let $b$ be a positive integer, and 
let $f$ be a positive integer-valued function on $V(G)$ with $2m\le f\le b$.
If $G$ is $b^2$-tough  and $|V(G)|\ge b^2$, then it has an $m$-tree-connected  $\{f, f+1\}$-factor. 
}\end{cor}
\begin{proof}
{We may assume that $G$ is a simple graph, by deleting multiple edges from $G$ (if necessary). 
By Theorem~\ref{intro:thm:f}, the graph $G$ has a factor $F$ such that for each vertex $v$, $d_F(v)=f(v)$, except possibly for a vertex $u$ with $d_F(u)=f(u)+1$.
By  Corollary~\ref{cor:f,f+1:lower-toughness}, the graph $G$ has an $m$-tree-connected  $\{f,f+1\}$-factor. 
}\end{proof}
\subsection{Tree-connected $\{a,a+1\}$-factors in $a$-tough graphs}
Enomoto, Jackson, Katerinis, and Saito (1985)~\cite{Enomoto-Jackson-Katerinis-Saito-1985} showed that every $r$-tough graph $G$ of order at least $r+1$ with $r|V(G)|$ even admits an $r$-factor. 
For the case that $r|V(G)|$ is odd, the same arguments can imply that the graph $G$ admits 
a factor whose degrees are $r$, except for a vertex with degree  $r+1$.
A combination of  Theorem~\ref{thm:tough:1-Extend} and this result can conclude the next results.
\begin{cor}\label{cor:toughness:r,r+1}
{Every $rt$-tough graph $G$ of order at least $r+1$ has an $m$-tree-connected $\{r,r+1\}$-factor, where 
$r\ge 2m$ and $t= \max\{1, \frac{2m}{r+1-2m}\}$.
}\end{cor}
\begin{proof}
{We may assume that $G$ is a simple graph, by deleting multiple edges from $G$ (if necessary). 
Let $F$ be a factor of $G$ such that each of whose vertices has degree $r$, except for at most one vertex $u$ with degree $r+1$~\cite{Enomoto-Jackson-Katerinis-Saito-1985}.
By applying Corollary~\ref{cor:f,f+1:lower-toughness}, 
the graph $G$ has an $m$-tree-connected factor  $\{r,r+1\}$-factor. 
}\end{proof}
\begin{cor}
{Every $4m^2$-tough graph of order at least $2m+1$ has an $m$-tree-connected $\{2m,2m+1\}$-factor.
}\end{cor}
\begin{proof}
{Apply Corollary~\ref{cor:toughness:r,r+1} with $r=2m$.
}\end{proof}
\begin{cor}{\rm(\cite{Ellingham-Nam-Voss-2002, Ellingham-Zha-2000}, see \cite{ClosedWalks})}\label{lem:r,r+1}
{Every $\max\{r,4\}$-tough graph of order at least $r+1$ admits a connected $\{r,r+1\}$-factor, 
where $r\ge 2$. 
}\end{cor}
\begin{proof}
{Apply Corollary~\ref{cor:toughness:r,r+1} with $m=1$.
}\end{proof}
\subsection{$2$-edge-connected $\{2f,2f+1\}$-factors}
\label{subsec:2-connected}
The following theorem gives a sufficient toughness condition for extending  
 $2$-factors with girth at least five to
$2$-connected $\{2,3\}$-factors. 
Ellingham and Zha~\cite{Ellingham-Zha-2000}  proved that 
 $2$-factors with girth at least three of $4$-tough graphs can be extended to connected $\{2,3\}$-factors. 
\begin{thm}\label{thm:2-edge-connected:tough}
{Let $G$ be a  graph with a factor $F$ of which  every component  is $2$-edge-connected and contains at least $c$ vertices with $c \ge 5$.
If for all $S\subseteq V(G)$,
$$\omega(G\setminus S) \le \frac{c-4}{4c-4}|S|+1,$$
then $G$ has a $2$-edge-connected factor $H$ containing $F$  such that for each vertex $v$, $d_H(v) \in \{d_F(v), d_F(v)+1\}$.
}\end{thm}
\begin{proof}
{Duplicate the edges of $F$ in $G$ and call the resulting graphs $F'$ and $G'$.
Obviously, every $2$-tree-connected component of $F'$ contains at least $c$ vertices.
By applying Theorem~\ref{thm:tough:1-Extend} with $m=2$, the graph $G'$ has a $2$-tree-connected factor $H'$ containing $F'$ such that for each vertex $v$, $d_{H'}(v)\le d_{F'}(v)+1$. 
Remove a copy of $F$ from $H'$ and call the resulting graph $H$.
It is easy to check that $H/F$ is still $2$-tree-connected and for each vertex $v$, $d_{H}(v)\le d_{F}(v)+1$.
By the assumption, every component of $F$ is  $2$-edge-connected.
Hence $H$ itself is $2$-edge-connected and the proof is completed.
}\end{proof}
\begin{cor}
{Let $G$ be a simple graph and let $f$ be an integer-valued function on $V(G)$ with $f\ge 2$. 
If $G$ has a $2f$-factor and for all $S\subseteq V(G)$, $\omega(G\setminus S)\le \frac{1}{16}|S|+1$,
then $G$ has a $2$-edge-connected $\{2f,2f+1\}$-factor including a $2f$-factor.
}\end{cor}
\begin{proof}
{Let $F$ be a $2f$-factor of $G$. Note that every component of $F$ must be Eulerian which is $2$-edge-connected.
Since $F$ is simple and has minimum degree at least four, every component of it contains at least five vertices.
Thus by applying Theorem~\ref{thm:2-edge-connected:tough} with $c=5$, the graph $F$ can be extended to a $2$-edge-connected $\{2f,2f+1\}$-factor $H$.
}\end{proof}
\begin{cor}\label{16:2-connected}
{Let $G$ be a simple graph having a $2$-factor $F$ with girth at least five. Then $G$ has a $2$-connected $\{2,3\}$-factor containing $F$, if for all $S\subseteq V(G)$, $\omega(G\setminus S)\le \frac{1}{16}|S|+1$.
}\end{cor}
\begin{proof}
{By applying Theorem~\ref{thm:2-edge-connected:tough} with $c=5$, 
the graph $F$ can be extended to a $2$-edge-connected $\{2,3\}$-factor $H$. Since $\Delta(H)\le 3$, the graph $H$ has no cut vertices, which can complete the proof.
}\end{proof}
\begin{cor}
{Every $16$-tough graph $G$ of girth at least five has a $2$-connected $\{2r,2r+1\}$-factor.
}\end{cor}
\begin{proof}
{We may assume that G is a simple graph, by deleting multiple edges from $G$ (if necessary). 
First, we choose a $2r$-factor  $F$ of $G$~\cite{Enomoto-Jackson-Katerinis-Saito-1985}, and next we choose a $2$-factor  $F_0$ of $F$. Since $F_0$ is simple,  it has girth at least five. Thus by Corollary~\ref{16:2-connected}, the factor $F_0$ can be extended to a $2$-connected $\{2,3\}$-factor $H$. It is easy to check that $H\cup F$ is the desired factor we are looking for.
}\end{proof}
\section{Highly tree-connected $\{f,f+k\}$-factors}
Our aim in this section is to give a sufficient toughness condition for the existence of  tree-connected $\{f,f+k\}$-factors.
For this purpose, we need to apply the second lemma in our proof, which shows that tough enough graphs can have sufficiently large bipartite  index.
\begin{lem}{\rm (\cite{ModuloFactorBounded})}
\label{lem:tree-connected:bipartite}\label{lem:bipartite:factor}
{Every $2m$-tree-connected graph  has an $m$-tree-connected bipartite factor.
}\end{lem}
\begin{lem}\label{lem:factor-bipartite-index}
{Let $G$ be a graph with a $(2k-2)$-tree-connected factor $F$ in which $|V(G)| \ge 4k-2$. 
If $G$ is $(4k-3)$-tough, then it has a matching $M$ of size $k-1$ such that  $bi(F\cup M)\ge k-1$.
}\end{lem}
\begin{proof}
{Let $k_0=k-1$ and $t=4k_0+1$.
By Lemma~\ref{lem:bipartite:factor}, there is a bipartition $X,Y$ of $V(G)$ such that $F[X,Y]$ is $k_0$-tree-connected.
We may assume that $|X|\ge |V(G)|/2$.
Let $M$ be a matching of $G[X]$ with the  maximum size so that  
$G[X]\setminus V(M)$ consists of isolated vertices. Define $S=Y\cup V(M)$. 
Since $t$-iso-tough,  we must  have 
$$|X|-2|E(M)|= \omega(G\setminus S)\le \frac{1}{t}|S|+1= \frac{1}{t}(|Y|+2|E(M)|)\le \frac{1}{t}(|X|+2|E(M)|),$$
which implies that
$|E(M)|  \ge \frac{t-1}{2t+2}|X|\ge \frac{t-1}{4t+4}|V(G)|\ge k_0$.
Let $e_{1}, \ldots, e_{k_0}$ be $k_0$ edges of $M$ and decompose $F[X,Y]$ into $k_0$ disjoint spanning trees $T_1,\ldots, T_{k_0}$.
Obviously, $T_i+e_i$ is not bipartite and so contains an odd cycle. This implies that
 $F\cup M$ has $k_0$ odd disjoint cycles, and hence it has the bipartite index at least $k_0$. Hence the proof is completed.
}\end{proof}
Now, are in a position to prove the main result of this section.
\begin{thm}\label{thm:f,f+k}
{Let $G$ be a  graph and  let $f$ be a positive integer-valued function on $V(G)$ 
satisfying $3m+2m_0+6k<  f +k\le b$ and  $m+m_0<k$, where $k$, $b$, $m$, and $m_0$ are four nonnegative integers.
If  $G$ is $4b^2$-tough  and $|V(G)|\ge 4b^2$, then $G$  has an $m$-tree-connected factor $H$ such that its complement is $m_0$-tree-connected and for each vertex $v$, $$d_H(v)\in \{f(v),f(v)+k\},$$
provided that $(k-1)\sum_{v\in V(G)}f(v)$ is even.
}\end{thm}
\begin{proof}
{For each vertex $v$, define $h(v)=f(v)+k-m-1$.
Since  $f+k>3m+2m_0+6k$, we must have  $2m' \le 2h(v) \le b'$, 
where $m'=2m+2m_0+6k$ and $b'=2b-2m-2$.
By the assumption, $|V(G)|\ge 4b^2\ge  (b')^2$.
Thus by Corollary~\ref{cor:f,f+1}, the graph $G$ has an $m'$-tree-connected $\{2h,2h+1\}$-factor $G'$.  
Since $|V(G)| \ge 4b^2 \ge 4k-2$ and $m'\ge 2k-2$, 
by Lemma~\ref{lem:factor-bipartite-index}, there is a  matching $M$ of size $k-1$ such that  $bi( G'\cup M)\ge k-1$. Let $G_0= G'\cup M$.
Note that for each vertex $v$,  $2h(v)\le d_{G_0}(v) \le 2h(v)+2$. 
Since $m+m_0<k$, we must have 
$f(v)+m_0\le h(v)\le \frac{1}{2}d_{G_0}(v)\le  h(v) +1= f(v)+k-m$.
Therefore, by Theorem~\ref{intro:thm:complementary}, the graph $G_0$ has an $m$-tree-connected $\{f,f+k\}$-factor $H$ such that its complement is $m_0$-tree-connected and so does $G$.
}\end{proof}
\section{Highly tree-connected $\{a,b\}$-factors}
When $f$ is a constant function,  Theorem~\ref{thm:f,f+k}  becomes simpler as the following version. 
By replacing Corollary~\ref{cor:toughness:r,r+1} in the proof, we also improve the needed toughness by  a linear bound.
\begin{thm}\label{thm:a-b}
{Let $G$ be a graph and let $a$ and $b$ be two positive integers  with $ab|V(G)|$ even.
Let $m$ and $m_0$ be two nonnegative integers with 
$$a +m+m_0< b < \frac{1}{5}(6a-3m-2m_0).$$ 
If $G$ is  $\max\{2b,256(b-a)^2\}$-tough and  $|V(G)|\ge 2b$, then $G$ admits an $m$-tree-connected $\{a,b\}$-factor such that its complement is $m_0$-tree-connected.
}\end{thm}
\begin{proof}
{Let $k=b-a$ and $h=b-m-1$.
Since  $b>3m+2m_0+6k$, we must have  $m' \le h$, 
where $m'=2m+2m_0+6k\ge 8k$.
By the assumption, $|V(G)|\ge 2b\ge  2h+1$.
Since $G$ is $\max\{2b, 4(m')^2\}$-tough, 
by applying Corollary~\ref{cor:toughness:r,r+1} with $r=2h$, the graph $G$ has an $m'$-tree-connected $\{2h,2h+1\}$-factor $G'$.  
Since $|V(G)| \ge 2b \ge 4k-2$ and $m'\ge 2k-2$, 
by Lemma~\ref{lem:factor-bipartite-index}, there is a  matching $M$ of size $k-1$ such that  $bi( G'\cup M)\ge k-1$. Let $G_0= G'\cup M$.
Note that for each vertex $v$,  $2h\le d_{G_0}(v) \le 2h+2$. 
Since $m+m_0<k$, we must have 
$a+m_0\le h\le \frac{1}{2}d_{G_0}(v)\le  h +1=b-m$.
Therefore, by Theorem~\ref{intro:thm:complementary}, the graph $G_0$ has an $m$-tree-connected $\{a,b\}$-factor $H$ such that its complement is $m_0$-tree-connected and so does $G$.
}\end{proof}

\begin{cor}\label{cor:a-b}
{Let $G$ be a graph and let $a$ and $b$ be two positive integers  with $ab|V(G)|$ even
and $a< b \le  \frac{1}{5}(6a+1)$. 
If $G$ is  $\max\{2b,64(b-a)^2\}$-tough and  $|V(G)|\ge b+1$, then $G$ admits a connected $\{a,b\}$-factor.
}\end{cor}
\begin{proof}
{If $|V(G)|\le 2b$, then $G$ must be the complete graph and so it is easy to see that it has a connected $r$-factor, where $r\in \{a,b\}$ and $r|V(G)|$ is even.
We may assume that $|V(G)|>2b$.
For the case $b=a+1$, the assertion follows from Corollary~\ref{cor:toughness:r,r+1}.
For the case $b\ge a+2$, the assertion follows  Theorem~\ref{thm:a-b}  by setting $m=1$ and $m_0=0$.
}\end{proof}
In the following, we are going to refine the special case $b=a+2$ of Corollary~\ref{cor:a-b}.
For this purpose,  we need to replace the following lemma in the proof. 
\begin{lem}{\rm (\cite{complementary})}\label{lem:complementary:k=2}
{Let $G$ be a $4$-tree-connected  graph and let $f$ be a positive integer-valued function on $V(G)$
 satisfying $\sum_{vin V(G)}f(v)\stackrel{2}{\equiv}2$. 
If  for each vertex $v$,  $2f(v)\le d_G(v) \le 2 f(v)+2$,
then $G$  has a connected factor $H$ such that for each vertex $v$,
$d_H(v)\in \{f(v),f(v)+2\}$.
}\end{lem}
The following theorem provides a natural generalization for Corollary~\ref{lem:r,r+1}.
\begin{thm}\label{thm:a-a+2}
{Every $\max\{2a,64\}$-tough graph $G$ of order at least $a+1$ with $a|V(G)|$ even admits a connected $\{a,a+2\}$-factor, where $a\ge 2$. 
}\end{thm}
\begin{proof}
{The special case $a=2$ was  proved in \cite{ClosedTrails}. First suppose that $a\ge 4$. 
If $|V(G)|\le 2a$, then $G$ must be  complete and so it admits a connected $a$-factor.
We may therefore assume that $|V(G)|> 2a$. Thus by Corollary~\ref{cor:toughness:r,r+1}, the graph $G$ has a $4$-tree-connected $\{2a,2a+1\}$-factor $G'$.  Now, by Lemma~\ref{lem:complementary:k=2}, the graph $G'$ has a connected $\{a,a+2\}$-factor and so dos $G$. Now,  suppose  that $a= 3$. Again,  by Corollary~\ref{cor:toughness:r,r+1}, the graph $G$ has a $3$-tree-connected $\{6,7\}$-factor $G'$. By counting the number of edges of $G'$, one can conclude that there is an edge $e$ of $G'$ such that $G'-e$  is $3$-tree-connected. This can  imply that $G'$ is $(2,1)$-partition-connected which means that it can be decomposed into a $2$-tree-connected factor $T$ and a factor $F$ having  an orientation with minimum out-degree at least $1$. Thus by Corollary~5.4 in~\cite{complementary},  the graph $G'$ has a connected $\{3,5\}$-factor and so dos $G$.  
}\end{proof}
\begin{cor}
{Every $\max\{4a,64\}$-tough graph $G$ of order at least $2a+1$  admits a spanning closed trail meeting each vertex $a$ or $a+1$ times, where $a\ge 1$. 
}\end{cor}
\begin{proof}
{By Theorem~\ref{thm:a-a+2}, the graph $G$ has a connected $\{2a,2a+2\}$-factor and so it admits a 
 spanning closed trail meeting each vertex $a$ or $a+1$ times.
}\end{proof}
%
%
%
%
%
%
%
%
%
%

\end{document}